\documentclass[reqno,12pt]{amsart}        %For Frank in the US
\usepackage{amsfonts}
\usepackage{colordvi}
\usepackage{amsmath,amssymb,amsthm} 
\usepackage{mathrsfs}
\usepackage{graphicx}
\theoremstyle{plain}      
\newtheorem{thm}{Theorem}[section]     
\newtheorem{theorem}[thm]{Theorem}     
     
\newtheorem{lemma}[thm]{Lemma}     
\newtheorem{proposition}[thm]{Proposition}

\theoremstyle{remark}      
 
\newtheorem{remark}[thm]{Remark} 
     
\theoremstyle{definition}      
\newtheorem{definition}[thm]{Definition}     

\newcommand{\wt}{\alpha}

\def\C{\mathbb{C}}

\def\K{\mathbb{K}}
\def\L{\mathbb{L}}

\def\Q{\mathbb{Q}}
\def\R{\mathbb{R}}
\def\Z{\mathbb{Z}}

\def\calA{\mathcal{A}}
\def\calV{\mathcal{V}}
\def\calW{\mathcal{W}}

\def\scrA{\mathscr{A}}
\def\scrC{\mathscr{C}}
\def\scrT{\mathscr{T}}

\DeclareMathOperator{\Log}{Log} % coordinatewise Logarithm map
\DeclareMathOperator{\trop}{trop} % interior of a ball

\let\witi\widetilde
\newcommand{\defcolor}[1]{\Blue{#1}}
\newcommand{\demph}[1]{\defcolor{{\sl #1}}}

\begin{document}     
%%%%%%%%%%%%%%%%%%%%%%%%%%%%%%%%%%%%%%%%%%%%%%%%%%%%%%%%%%%%%%%%%%%%%%%%%

%%%%%%%%%%%%%%%%%%%%%%%%%%%%%%%%%%%%%%%%%%%%%%%%%%%%%%%%%%%%%%%%%%%%%%%%%

%%%%%%%%%%%%%%%%%%%%%%%%%%%%%%%%%%%%%%%%%%%%%%%%%%%%%%%%%%%%%%%%%%%%%%%%%

\title[Higher convexity of tropical varieties]{Higher convexity for complements\\ of tropical varieties}  
 
%%%%%%%%%%%%%%%%%%%%%%%%%%%%%%%%%%%%%%%%%%%%%%%%%%%%%%%%%%%%%%%%%%%%%%%%%%%%
\author{Mounir Nisse}
\address{School of Mathematics KIAS, 87 Hoegiro Dongdaemun-gu, Seoul
130-722, South Korea.}
\email{mounir.nisse@gmail.com}
%\urladdr{}
%\thanks{}
%%%%%%%%%%%%%%%%%%%%%%%%%%%%%%%%%%%%%%%%%%%%%%%%%%%%%%%%%%%%%%%%%%%%%%%%%%%%
\author{Frank Sottile}
\address{Department of Mathematics\\
         Texas A\&M University\\
         College Station\\
         Texas\\
         USA}
\email{sottile@math.tamu.edu}
\urladdr{www.math.tamu.edu/\~{}sottile}
\thanks{Research of Sottile is supported in part by NSF grant DMS-1001615.}
\thanks{This is based upon work done at the NIMS, Daejeon, Korea, This work was supported by National Institute
for Mathematical Sciences 2014 Thematic Program} 
%%%%%%%%%%%%%%%%%%%%%%%%%%%%%%%%%%%%%%%%%%%%%%%%%%%%%%%%%%%%%%%%%%%%%%%%%%%%
%
% 14T05 Tropical Geometry
% 32A60 Zero sets of homomorphic functions
%
%%%%%%%%%%%%%%%%%%%%%%%%%%%%%%%%%%%%%%%%%%%%%%%%%%%%%%%%%%%%%%%%%%%%%%%%%%%%
%\date{\today}

\begin{abstract}
 We consider Gromov's homological higher convexity for complements of tropical varieties, establishing it 
 for complements of tropical hypersurfaces and 
 curves, and for nonarchimedean amoebas of varieties that are complete intersections over the field of complex
 Puiseux series.  
 Based on these results, we conjecture that the complement of a tropical variety has this higher convexity, and 
 prove a weak form of this conjecture for the nonarchimedean amoeba of any variety over the
 complex Puiseux field.
 One of our main tools is Jonsson's limit theorem for tropical varieties.
\end{abstract}

\maketitle

%%%%%%%%%%%%%%%%%%%%%%%%%%%%%%%%%%%%%%%%%%%%%%%%%%%%%%%%%%%%%%%%%%%%%%%%%

A tropical hypersurface is a polyhedral complex in $\R^n$ of pure dimension $n{-}1$ that is dual to a
regular subdivision of a finite set of integer vectors.
This implies that every connected component of its complement is convex.
A classical (archimedean) amoeba of a complex hypersurface also has the property that every 
connected component of its complement is convex~\cite[Ch.~6, Cor.~1.6]{GKZ}.

%To understand the complement of general amoebas, 
Gromov~\cite[\S$\tfrac{1}{2}$]{Gromov} introduced  higher convexity.
A subset $X\subset\R^n$ is \demph{$k$-convex} if for all affine planes $L$ of dimension $k{+}1$, the natural  
map on $k$th reduced homology
 \begin{equation}\label{Eq:iota_k}
   \iota_k\ \colon\ \witi{H}_k(X\cap L)\ \longrightarrow\ 
     \witi{H}_k(X)
 \end{equation}
is an injection.
Connected and 0-convex is ordinary convexity.
Henriques~\cite{He04} rediscovered this notion and conjectured that the complement of an amoeba of a variety of
codimension $k{+}1$ in $(\C^\times)^n$ is $k$-convex, and established a weak form of this conjecture: the map
$\iota_k$ sends no positive class to zero~\cite{He04}. 
Bushueva and Tsikh~\cite{BT12} used complex analysis to prove Henriques' conjecture when
the variety is a complete intersection. 
%
%  Rashkovskii ????
%
Other than these cases, Henriques' conjecture remains open.

An amoeba is the image in $\R^n$ of a subvariety $V$ of the torus $(\C^\times)^n$ under the coordinatewise
map $z\mapsto \log|z|$.
Similarly, the coamoeba is the image in $(S^1)^n$ of $V$ under the coordinatewise argument map.
Lifting to the universal cover and taking closure gives the lifted coamoeba in $\R^n$.
There is also a nonarchimedean coamoeba and a lifted nonarchimedean coamoeba~\cite{NSna}.
The complement of either type of lifted coamoeba of a variety of codimension $k{+}1$ is
$k$-convex~\cite{NSconvex}, which was proven using tropical geometry. 

We investigate Gromov's higher convexity for complements of tropical varieties.
We show that the complement of a tropical curve in $\R^n$ is 
$(n{-}2)$-convex.
Both this and the convexity of tropical hypersurface complements rely only on some
properties of tropical varieties and hold for polyhedral complexes having these properties.
This leads us to conjecture that the complement of a tropical variety of pure codimension $k{+}1$ is
$k$-convex. 

Suppose that $\K$ is an algebraically closed field having a nonarchimedean valuation with
value group dense in $\R$ and let $V\subset(\K^\times)^n$ be a subvariety. 
The closure in $\R^n$ of the image of $V$ under the coordinatewise valuation map is its nonarchimedean
amoeba, which is a tropical variety.
When $\K$ is the field of complex Puiseux series, methods from analysis, both archimedean and
nonarchimedean, may be used to study nonarchimedean amoebas.
Jonsson showed that such a nonarchimedean amoeba is a limit of archimedean amoebas~\cite[Thm.~B]{J14},
generalizing work of Rullg{\aa}rd~\cite[Thm.~9]{Ru01} and Mikhalkin~\cite[Cor.~6.4]{Mi04} for
hypersurfaces. 
This, together with a technical lemma, allows us to use the result of Bushueva and Tsikh~\cite{BT12} to establish
our conjecture for the nonarchimedean amoeba of a complete intersection in $(\K^\times)^n$.
The weak form (that the map $\iota_k$~\eqref{Eq:iota_k} sends no positive cycle to zero) for any subvariety of
$(\K^\times)^n$ also follows by Henriques's result for amoebas.

This paper is organized as follows.
In Section~\ref{S:one} we give background material.
In Section~\ref{S:two}, we define combinatorial tropical hypersurfaces and curves in $\R^n$, and prove that their
complements are $0$-convex and $(n{-}2)$-convex, respectively.
In Section~\ref{S:three}, we prove our
results about nonarchimedean amoebas stated above.

%%%%%%%%%%%%%%%%%%%%%%%%%%%%%%%%%%%%%%%%%%%%%%%%%%%%%%%%%%%%%%%%%%%%%%%%%%%%%%%%%
\section{Background}\label{S:one}
%%%%%%%%%%%%%%%%%%%%%%%%%%%%%%%%%%%%%%%%%%%%%%%%%%%%%%%%%%%%%%%%%%%%%%%%%%%%%%%%%

We provide some background on tropical varieties, state Jonsson's limit theorem, and discuss higher
convexity.  

%%%%%%%%%%%%%%%%%%%%%%%%%%%%%%%%%%%%%%%%%%%%%%%%%%%%%%%%%%%%%%%%%%%%%%%%%%%%%%%%%
\subsection{Tropical varieties}
The map $z\mapsto \log|z|$ is a homomorphism from the non-zero complex numbers $\C^\times$ to the real
numbers $\R$. 
This induces the map $\Log\colon(\C^\times)^n\to\R^n$.
The image of subvariety $V$ of $(\C^\times)^n$ under
$\Log$ is its \demph{amoeba}, \defcolor{$\scrA(V)$}.

Let $\K$ be an algebraically closed valued field whose value group $G$ is a non-zero divisible additive
subgroup of $\R$. 
Its valuation is a surjective homomorphism $\nu\colon\K^\times\to G$, which induces a map
$\nu\colon(\K^\times)^n\to G^n$. 
The closure in $\R^n$ of the image of a variety $V\subset(\K^\times)^n$ under the map $\nu$ is its
\demph{nonarchimedean amoeba}, \defcolor{$\scrT(V)$}.

There is an equivalent definition of $\scrT(V)$.
An integer vector $a\in\Z^n$ forms the exponents of a Laurent monomial,
$\defcolor{x^a}:=x_1^{a_1}\dotsb x_n^{a_n}$.
A \demph{Laurent polynomial $f$} is a $\K$-linear combination of Laurent monomials,
\[
   f\ =\ \sum_{a\in\calA} c_a x^a
   \qquad\mbox{where}\qquad
   c_a\in\K^\times\,.
\]
The finite subset $\calA\subset\Z^n$ is the \demph{support} of $f$.
The coordinate ring of $(\K^\times)^n$ is the ring of Laurent polynomials
$\K[x_1,x_1^{-1},\dotsc,x_n,x_n^{-1}]$.
Given a vector $w\in\R^n$ and a Laurent polynomial $f$ with support $\calA\subset\Z^n$,
we have a piecewise linear map 
 \begin{equation}\label{Eq:Trop_Fn}
   \R^n\ni z\ \longmapsto\ 
   \min\{ \nu(c_a)+w\cdot a\,|\, a\in\calA\}\,.
 \end{equation}
The tropical hypersurface \demph{$\trop(f)$} is the set where this minimum occurs at least twice.
This is also the set where the piecewise linear map~\eqref{Eq:Trop_Fn} is not differentiable. 
Given a variety $V\subset(\K^\times)^n$, let $I$ be its ideal in the ring of Laurent
polynomials.
Its tropical variety is
\[
   \bigcap_{f\in I} \trop(f)\,.
\]
By the fundamental theorem of tropical geometry~\cite{MaSt}, this
tropical variety equals its nonarchimedean amoeba, $\scrT(V)$.
If $V$ has dimension $r$, then $\scrT(V)$ admits (non-canonically) the
structure of a polyhedral complex of pure dimension $r$.
There are positive integral weights \demph{$\wt_\sigma$} associated to each polyhedron $\sigma$ of
maximal dimension $r$ so that the weighted complex is \demph{balanced}.
We explain this.
Let $\tau$ be an $(r{-}1)$-dimensional polyhedron in $\scrT(V)$.
Modulo the affine span \defcolor{$\langle\tau\rangle$} of $\tau$, each
$r$-dimensional polyhedron $\sigma$ incident 
on $\tau$  ($\sigma\in\mbox{star}(\tau)$) determines a primitive vector \defcolor{$v_\sigma$}.
The balancing condition is that 
\[
   \sum_{\sigma\in\mbox{\scriptsize star}(\tau)} \wt_\sigma v_\sigma\ =\ 0\ \ \mod \langle\tau\rangle\,.
\]

We are primarily concerned with nonarchimedean amoebas when the field $\K$ is the complex
Puiseux field.
This algebraically closed field contains the field of rational functions $\C(s)$, which is the quotient field of
the ring of univariate polynomials $\C[s]$.
An element of the Puiseux field is a fractional power series of the form
\[
   c\ =\ \sum_{m\geq 0} a_m s^{\frac{p+m}{q}}\,,
\]
where the coefficients $a_m$ are complex numbers and $p,q$ are integers with $q>0$.
The valuation $\nu(c)$ is the minimum exponent appearing in the power series $c$ with
a non-zero coefficient.
This is a homomorphism $\nu\colon\K^\times\to\Q$.

%%%%%%%%%%%%%%%%%%%%%%%%%%%%%%%%%%%%%%%%%%%%%%%%%%%%%%%%%%%%%%%%%%%%%%%%%%%%%%%%%
\subsection{Jonsson's limit theorem}

Let \demph{$d(x,y)$} be the Euclidean distance in $\R^n$ between $x$ and $y$.
If $\emptyset\neq A\subset\R^n$ is closed and $x\in\R^n$, then the distance between $x$ and $A$ is 
\[
   \defcolor{d(x,A)}\ :=\ \inf\{d(x,a)\mid a\in A\}\,,
\]
which is attained as $A$ is closed.
If $B$ is closed, then 
$d(A,B)=\inf\{d(a,b)\mid a\in A, b\in B\}$.
A family $\{A_t\mid t>0\}$ of closed subsets of $\R^n$ has \demph{Kuratowski limit} $\scrT\subset\R^n$ if we have
the following equality,
 \begin{eqnarray}
  \scrT&=& 
   \{x\in\R^n\mid \forall\epsilon>0\; \exists\delta>0\ \mbox{such that}\  \nonumber %\label{Eq:LP}
             0<t<\delta \Rightarrow d(x,A_t)<\epsilon\}\\
   &=&   \{x\in\R^n\mid \forall\epsilon>0\; \forall\delta>0\; \label{Eq:AP}
             \exists 0<t<\delta\ \mbox{with}\  d(x,A_t)<\epsilon\}\,.
 \end{eqnarray}
When this occurs, we write
\[
   \lim_{t\to0} A_t\ =\ \scrT\,.
\]
This is distinct from the more familiar notion of Hausdorff limit:
Consider the family of lines in the plane $\R^2$ through the origin with slope $t>0$.
As $t\to 0$, these lines have Kuratowski limit the $x$-axis, but they do not have a Hausdorff limit.

Observe that the Kuratowski limit $\scrT$ is closed.
More interesting is if $Z$ is a compact set disjoint from $\scrT$.

%%%%%%%%%%%%%%%%%%%%%%%%%%%%%%%%%%%%%%%%%%%%%%%%%%%%%%%%%%%%%%%%%%%%%%%%%%%%%%%%%
\begin{lemma}\label{L:compacta}
 Suppose that $\{A_t\mid t>0\}$ is a family of closed subsets of $\R^n$  with Kuratowski limit $\scrT$.
 If $Z$ is a compact subset of\/ $\R^n$ that is disjoint from $\scrT$, then there is a $\delta>0$ such that if
 $0<t<\delta$, then $Z\cap A_t=\emptyset$.
\end{lemma}
%%%%%%%%%%%%%%%%%%%%%%%%%%%%%%%%%%%%%%%%%%%%%%%%%%%%%%%%%%%%%%%%%%%%%%%%%%%%%%%%%

%%%%%%%%%%%%%%%%%%%%%%%%%%%%%%%%%%%%%%%%%%%%%%%%%%%%%%%%%%%%%%%%%%%%%%%%%%%%%%%%%
\begin{proof}
 By~\eqref{Eq:AP}, if $z\in Z$, so that $z\not\in\scrT$, then there are $\epsilon,\delta>0$ such that for every
 $0<t<\delta$, we have $d(z,A_t)>\epsilon$.
 Since $Z$ is compact, there are $\epsilon,\delta>0$ such that if $z\in Z$ and $0<t<\delta$, then 
 $d(z,A_t)>\epsilon$.
 In particular, $0<t<\delta$ implies that $Z\cap A_t=\emptyset$.
\end{proof}
%%%%%%%%%%%%%%%%%%%%%%%%%%%%%%%%%%%%%%%%%%%%%%%%%%%%%%%%%%%%%%%%%%%%%%%%%%%%%%%%%

Let $\calV\subset\C^\times\times(\C^\times)^n$ be a subvariety whose every component maps
dominantly onto the first factor, $\C^\times$, which has coordinate $s$.
Then $\calV$ is a family of varieties over an open subset $U$ of $\C^\times$, with fiber $\calV_s$ over
$s\in U$.
Equivalently, $\calV$ is a variety in the torus $(\C(s)^\times)^n$ over $\C(s)$.
Extending scalars to the Puiseux field $\K$ gives a variety
$V\subset(\K^\times)^n$ with tropicalization $\scrT(V)$.
In this context, Jonsson~\cite{J14} proved the following.

%%%%%%%%%%%%%%%%%%%%%%%%%%%%%%%%%%%%%%%%%%%%%%%%%%%%%%%%%%%%%%%%%%%%%%%%%%%%%%%%%
\begin{proposition}[Jonsson]\label{P:Jonsson}
 We have 
 ${\displaystyle  \lim_{s\to 0}\tfrac{-1}{\log|s|} \scrA(\calV_s)\ =\ \scrT(V)}$.
\end{proposition}
%%%%%%%%%%%%%%%%%%%%%%%%%%%%%%%%%%%%%%%%%%%%%%%%%%%%%%%%%%%%%%%%%%%%%%%%%%%%%%%%%

For this, $\nu(s)=1$.
More generally, $\scrT(V)$ should be scaled by $\nu(s)$.
To summarize Jonsson's Theorem, the nonarchimedean amoeba of $V$ is the limit of (appropriately scaled) amoebas of
fibers of the family $\calV$.
Stated in this way, Jonsson's Theorem holds in the broader context of tropicalizations of varieties in
$(\K^\times)^n$. 

%%%%%%%%%%%%%%%%%%%%%%%%%%%%%%%%%%%%%%%%%%%%%%%%%%%%%%%%%%%%%%%%%%%%%%%%%%%%%%%%%
\begin{proposition}\label{P:limit}
 Let $W\subset(\K^\times)^n$ be any variety.
 Then there is a smooth curve $C$, a point $o\in C$, a local parameter $u$ at $o$, and a family of varieties
 $\calV\subset(C\smallsetminus\{o\})\times(\C^\times)^n$ over $C\smallsetminus\{o\}$ with fiber $\calV_a$ over
 $a\in C\smallsetminus\{o\}$ such that
\[
    \lim_{a\to o} \tfrac{-1}{\log|u(a)|} \scrA(\calV_a)\ =\ \scrT(W)\,.
\]
 If $W$ is a complete intersection, then we may choose the family $\calV$ so that every fiber $\calV_a$ is a
 complete intersection.
\end{proposition}
%%%%%%%%%%%%%%%%%%%%%%%%%%%%%%%%%%%%%%%%%%%%%%%%%%%%%%%%%%%%%%%%%%%%%%%%%%%%%%%%%

The proof we give uses the following result of Katz.

%%%%%%%%%%%%%%%%%%%%%%%%%%%%%%%%%%%%%%%%%%%%%%%%%%%%%%%%%%%%%%%%%%%%%%%%%%%%%%%%%
\begin{proposition}[Thm.~1.5~\cite{Katz}]\label{P:Katz}
 Let $W$ be a variety in $(\K^\times)^n$.
 Then there is a finite extension $\L$ of\/ $\C(s)$ and a variety $W'$ in $(\L^\times)^n$
 with the same tropicalization and the same Hilbert polynomial as $W$.
\end{proposition}
%%%%%%%%%%%%%%%%%%%%%%%%%%%%%%%%%%%%%%%%%%%%%%%%%%%%%%%%%%%%%%%%%%%%%%%%%%%%%%%%%

This is discussed in the paragraph following the statement of Theorem~1.5 in~\cite{Katz}.

%%%%%%%%%%%%%%%%%%%%%%%%%%%%%%%%%%%%%%%%%%%%%%%%%%%%%%%%%%%%%%%%%%%%%%%%%%%%%%%%%
\begin{proof}[Proof of Proposition~$\ref{P:limit}$]
 By Proposition~\ref{P:Katz}, we may assume that $W$ is defined over a finite extension of $\C(s)$, which  
 is the function field, $\C(C)$, of a smooth complex curve $C$.
 The inclusion $\C(s)\subset\C(C)$ induces a dominant rational map $\pi\colon C\,-\to\C$.
 If we let $o\in\pi^{-1}(0)$ and $u$ be a local parameter at $o$, then $u$ generates $\C(C)$ over $\C(s)$.
 
 Replacing $C$ by an affine neighborhood of $o$ if necessary, there is a family of varieties
 $\calW\subset(C\smallsetminus\{o\})\times(\C^\times)^n$ over $C\smallsetminus\{o\}$ such that, when scalars are
 extended to $\K$, gives the variety $W\subset(\K^\times)^n$.
 We claim that 
\[
     \lim_{a\to o} \tfrac{-1}{\log|u(a)|} \scrA(\calW_a)\ =\ \nu(u)\cdot\scrT(W)\,.
\]
 This follows from nearly the same arguments as Proposition~\ref{P:Jonsson}, which are given in Section 4
 of~\cite{J14}. 
 To obtain the statement of the proposition, set $\calV:=e^{-\nu(u)}\calW$.

 Finally, if the original variety $W\subset(\K^\times)^n$ was a complete intersection then so is the variety we
 replace it by in $(\C(C)^\times)^n$ as they have the same Hilbert polynomial.
 Thus there is an open subset $U$ of $C$ such that $\calW$ is flat over $U$, and in fact so that $\calW_a$ is
 a complete intersection for $a\in U$.
\end{proof}
%%%%%%%%%%%%%%%%%%%%%%%%%%%%%%%%%%%%%%%%%%%%%%%%%%%%%%%%%%%%%%%%%%%%%%%%%%%%%%%%%

%%%%%%%%%%%%%%%%%%%%%%%%%%%%%%%%%%%%%%%%%%%%%%%%%%%%%%%%%%%%%%%%%%%%%%%%%%%%%%%%%
\subsection{Higher convexity}

Let $Y\subset\R^{k+1}$ be a closed set with finitely many connected components such that the pair $(\R^{k+1},Y)$ is
triangulated. 
Set $\defcolor{Y^c}:=\R^{k+1}\smallsetminus Y$.
Let $Y_1,\dotsc,Y_m$ be the bounded connected components of $Y$.
Then there exist open subsets $\beta_1,\dotsc,\beta_m$ of $\R^{k+1}$ with disjoint closures, where $\beta_i$ a
neighborhood of $Y_i$ and $\gamma_i:=\partial \beta_i$, the boundary of $\beta_i$, is a $k$-cycle in
$Y^c$. 

%%%%%%%%%%%%%%%%%%%%%%%%%%%%%%%%%%%%%%%%%%%%%%%%%%%%%%%%%%%%%%%%%%%%%%%%%%%%%%%%%
\begin{lemma}\label{L:homology_basis}
 The classes of the cycles $\gamma_1,\dotsc,\gamma_m$ form a basis for $\witi{H}_{k}(Y^c)$.
\end{lemma}
%%%%%%%%%%%%%%%%%%%%%%%%%%%%%%%%%%%%%%%%%%%%%%%%%%%%%%%%%%%%%%%%%%%%%%%%%%%%%%%%%

We remark that we take coefficients in $\Z$.
Fixing an orientation for $\R^{k+1}$ orients each $\beta_i$ and gives each $\gamma_i=\partial\beta_i$ the
outward orientation. 
A nonzero homology class $\zeta\in\witi{H}_k(Y^c)$ is \demph{positive} if it is a nonnegative
integer combination of the classes $[\gamma_i]$.

A cycle class $[\gamma]\in\witi{H}_k(Y^c)$ vanishes in $\witi{H}_k(\R^{k+1})$, and so it is
the boundary of a $(k{+}1)$-chain $Z$ in $\R^{k+1}$.
Then $[\gamma]\in\witi{H}_k(Y^c)$ is non-zero if and only if 
for every $(k{+}1)$-chain $Z$ with $\partial Z=\gamma$, we have $Z\cap Y\neq\emptyset$.

Gromov~\cite[\S$\tfrac{1}{2}$]{Gromov} gave a homological generalization of convexity.
An open set $X\subset\R^n$ is \demph{$k$-convex} if for any affine linear space $L$ of dimension $k{+}1$, the
natural map 
\[
   \iota_k\ \colon\ \witi{H}_k(L\cap X)\ \longrightarrow\ 
     \witi{H}_k(X)
\]
is an injection.
Note that $0$-convex and connected is the ordinary notion of convex.
Henriques~\cite{He04} considered bu did not define a weak version of this notion:
The set $X$ is \demph{weakly $k$-convex} if the map $\iota_k$ does not send any positive cycle to zero.
This is independent of choices,  changing the orientation of $L$ replaces each positive
cycle $\gamma$ by $-\gamma$.

Let $Y$ be a polyhedral complex of codimension $k{+}1$ in $\R^n$ consisting of finitely many polyhedra.
An affine subspace $L$ in $\R^n$ of dimension $k{+}1$ is \demph{transverse to $Y$} if $L$ meets each polyhedron
$\sigma$ of $Y$ transversally.
That is, if $L\cap\sigma\neq\emptyset$, then $\sigma$ has dimension $n{-}k{-}1$ and $L$ meets it transversally,
necessarily in a single point in the relative interior of $\sigma$.

Henriques proved a moving lemma~\cite[Lemma~3.6]{He04}, which implies that it suffices to take the affine
space $L$ in the definition of $k$-convexity to be a translate of some rational affine space---one of the
form $M\otimes_\Q\R$, for $M$ an affine subspace of $\Q^n$ of dimension $k{+}1$.
The same proof shows that it suffices to take $L$ to lie in a dense subset of such subspaces.

%%%%%%%%%%%%%%%%%%%%%%%%%%%%%%%%%%%%%%%%%%%%%%%%%%%%%%%%%%%%%%%%%%%%%%%%%%%%%%%%%
\section{Combinatorial tropical varieties}\label{S:two}
%%%%%%%%%%%%%%%%%%%%%%%%%%%%%%%%%%%%%%%%%%%%%%%%%%%%%%%%%%%%%%%%%%%%%%%%%%%%%%%%%

A ``combinatorial tropical variety'' of dimension $r$ is a polyhedral complex in $\R^n$ of pure dimension
$r$ in $\R^n$ that has some of the properties of a nonarchimedean amoeba.
Mikhalkin and Rau~\cite{MR} call these tropical cycles.
Consequently, a result about some type of combinatorial tropical variety implies the same result for nonarchimedean
amoebas.
Let $\calA\subset\R^n$ be a finite set of points and $c\in\R^\calA$ a vector whose components $c_a$ are
indexed by elements $a$ of $\calA$.
These define a piecewise linear function on $\R^n$,
 \begin{equation}\label{Eq:ATV}
  \defcolor{T(\calA,c)}\ :=\  
    x\ \longmapsto\ \min\{c_a+w\cdot a\mid a\in\calA\}\,.
 \end{equation}
Its graph \defcolor{$\Gamma(\calA,c)$} is a polyhedral complex of dimension $n$ whose facets lie over the domains
of linearity of $T(\calA,c)$. 
Its ridge set is the union of faces of dimension at most $n{-}1$, which lies over the set  
\defcolor{$\scrT(\calA,c)$} where $T(\calA,c)$ is not differentiable.
This set $\scrT(\calA,c)$ is a \demph{combinatorial tropical hypersurface} and it consists of
the points $x\in\R^n$ where the minimum in~\eqref{Eq:ATV} is attained at least twice.
The following is elementary and not original.

%%%%%%%%%%%%%%%%%%%%%%%%%%%%%%%%%%%%%%%%%%%%%%%%%%%%%%%%%%%%%%%%%%%%%%%%%%%%%%%%%
\begin{proposition}\label{P:trivial}
 The set $\scrT(\calA,c)$ is a polyhedral complex of pure dimension $n{-}1$ whose complement is $0$-convex. 
\end{proposition}
%%%%%%%%%%%%%%%%%%%%%%%%%%%%%%%%%%%%%%%%%%%%%%%%%%%%%%%%%%%%%%%%%%%%%%%%%%%%%%%%%

%%%%%%%%%%%%%%%%%%%%%%%%%%%%%%%%%%%%%%%%%%%%%%%%%%%%%%%%%%%%%%%%%%%%%%%%%%%%%%%%%
\begin{proof}
 The projection of $\Gamma(\calA,c)$ to $\R^n$ is a piecewise linear homeomorphism.
 The first statement follows as its ridge set has pure dimension $n{-}1$ and the second as the components of
 the complement of the ridge set consists of the interiors of the facets of $\Gamma(\calA,c)$.
\end{proof}
%%%%%%%%%%%%%%%%%%%%%%%%%%%%%%%%%%%%%%%%%%%%%%%%%%%%%%%%%%%%%%%%%%%%%%%%%%%%%%%%%

When $\calA\subset\Z^n$, the set $\scrT(A,c)$ is the tropicalization of the hypersurface
\[
   \{  x\in(\K^\times)^n \,\mid\, 0 = \sum_{a\in\calA} s^{c_a}x^a\ \}\,,
\]
where $\K$ is a valued field with value group $\R$ and $s\in\K$ has $\nu(s)=1$.

By the fundamental theorem of tropical geometry~\cite{SS}, the nonarchimedean amoeba
$\scrC$ of a curve $C$ in $(\K^\times)^n$ admits the structure of a finite balanced rational polyhedral complex in
$\R^n$ of pure 
dimension one. 
Putting such a structure on $\scrC$, we have that $\scrC$ consists of finitely many vertices and edges,
with each edge an interval (possibly unbounded) of a line.
Furthermore, each edge $e$ is equipped with a positive integral wight $\wt_e$ and is parallel to a vector
$v_e\in\Z^n$. 
We assume that $v_e$ is primitive in that its components are relatively prime.
There are exactly two primitive vectors, $v_e$ and $-v_e$, that are parallel to $e$, one for each
direction along $e$.
Finally, balanced means that for every point $p\in\scrC$, we have
 \begin{equation}\label{Eq:balanced}
   0\ =\ \sum_{e} \wt_e v_e\,,
 \end{equation}
the sum over all edges $e$ incident on $p$ where $v_e$ points away from $p$.
This sum~\eqref{Eq:balanced} is nonempty as $\scrC$ is pure and therefore has no isolated points. 

%%%%%%%%%%%%%%%%%%%%%%%%%%%%%%%%%%%%%%%%%%%%%%%%%%%%%%%%%%%%%%%%%%%%%%%%%%%%%%%%%
\begin{lemma}\label{L:weaklyBalanced}
 Suppose that $w\in\R^n$ and that $p\in\scrC$.
 Then either we have that $w\cdot v_e=0$ for all edges $e$ incident on $p$, or else there are two edges
 $e,f$ incident on $p$ with 
\[
   w\cdot v_e\ <\ 0\ <\ w\cdot v_f\,.
\]
\end{lemma}
%%%%%%%%%%%%%%%%%%%%%%%%%%%%%%%%%%%%%%%%%%%%%%%%%%%%%%%%%%%%%%%%%%%%%%%%%%%%%%%%%

%%%%%%%%%%%%%%%%%%%%%%%%%%%%%%%%%%%%%%%%%%%%%%%%%%%%%%%%%%%%%%%%%%%%%%%%%%%%%%%%%
\begin{proof}
 This follows from~\eqref{Eq:balanced} as each weight $\wt_e$ is positive.
\end{proof}
%%%%%%%%%%%%%%%%%%%%%%%%%%%%%%%%%%%%%%%%%%%%%%%%%%%%%%%%%%%%%%%%%%%%%%%%%%%%%%%%%

%%%%%%%%%%%%%%%%%%%%%%%%%%%%%%%%%%%%%%%%%%%%%%%%%%%%%%%%%%%%%%%%%%%%%%%%%%%%%%%%%
\begin{definition}
 A locally finite polyhedral complex $\scrC$ in $\R^n$ of pure dimension one is 
 \demph{weakly balanced} if Lemma~\ref{L:weaklyBalanced} holds for $\scrC$.
 That is, if for all $w\in\R^n$ and $p\in\scrC$, 
 either $w\cdot v_e=0$ for all edges $e$ incident on $p$ or there are two edges
 $e,f$ incident on $p$ with 
\[
   w\cdot v_e\ <\ 0\ <\ w\cdot v_f\,.
\]
 Here, $v_e$ and $v_f$ are any vectors pointing away from $p$ that are parallel to $e$ and $f$, respectively.
\end{definition}
%%%%%%%%%%%%%%%%%%%%%%%%%%%%%%%%%%%%%%%%%%%%%%%%%%%%%%%%%%%%%%%%%%%%%%%%%%%%%%%%%

Weakly balanced graphs admit unbounded paths in nearly every direction.

%%%%%%%%%%%%%%%%%%%%%%%%%%%%%%%%%%%%%%%%%%%%%%%%%%%%%%%%%%%%%%%%%%%%%%%%%%%%%%%%%
\begin{lemma}\label{L:unboundedPath}
 Let $\scrC$ be a weakly balanced graph and $w\in\R^n$ be nonzero.
 Then, for every point $p$ of $\scrC$ which has an incident edge $e$ with $w\cdot v_e\neq 0$, 
 there is a continuous path $\gamma\colon [0,\infty)\to\scrC$ with $\gamma(0)=p$ where 
 $w\cdot\gamma(t)$ is unbounded and strictly increasing for $t\in[0,\infty)$, and we may assume that $\gamma$
 contains an edge $f$ incident to $p$ with $w\cdot v_f>0$.
\end{lemma}
%%%%%%%%%%%%%%%%%%%%%%%%%%%%%%%%%%%%%%%%%%%%%%%%%%%%%%%%%%%%%%%%%%%%%%%%%%%%%%%%%

In particular, every component of a weakly balanced graph is unbounded.

%%%%%%%%%%%%%%%%%%%%%%%%%%%%%%%%%%%%%%%%%%%%%%%%%%%%%%%%%%%%%%%%%%%%%%%%%%%%%%%%%
\begin{proof}
 Let us assume that if $w\cdot q>w\cdot p$ and $q$ is a vertex of $\scrC$ having an incident edge $e$ with
 $w\cdot v_e\neq 0$, then there is a continuous function $\gamma_1\colon[0,\infty)\to\scrC$ with $\gamma(0)=q$
 such that $w\cdot\gamma_1(t)$ is an increasing unbounded function. 
 The base case of the induction are those vertices $p$ of $\scrC$ with $w\cdot p$ maximal, which is
 covered later in this paragraph as edges $e$ emanating from $p$ with $w\cdot v_e>0$ are unbounded.
 Let $p\in\scrC$ be a point with an incident edge $f$ such that $w\cdot v_f\neq 0$.
 Since $\scrC$ is weakly balanced, there is an edge $e$ incident to $p$ with $w\cdot v_e>0$.
 If $e$ is unbounded in the direction of $v_e$, set $\gamma(t):=p+t v_e$, which gives the desired path.

 Otherwise, let $q$ be the vertex incident to $e$ in the direction of $v_e$.
 Then $w\cdot q>w\cdot p$ and the direction of $e$ at $q$ is $-v_e$.  
 Since $w\cdot(-v_e)\neq 0$, the induction hypothesis holds. 
 Let $\gamma_1\colon[0,\infty)\to\scrC$ be a path with $\gamma_1(0)=q$ and $w\cdot \gamma_1(t)$
 increasing and unbounded.
 Suppose that $p+t_0v_e=q$.  
 Then $t_0>0$ and we define $\gamma(t)$ by $\gamma(t):=p+t v_e$ for $0\leq t\leq t_0$ and 
 $\gamma(t)=\gamma_1(t-t_0)$ for $t\geq t_0$.  
 Then $\gamma$ is the desired increasing path.
\end{proof}
%%%%%%%%%%%%%%%%%%%%%%%%%%%%%%%%%%%%%%%%%%%%%%%%%%%%%%%%%%%%%%%%%%%%%%%%%%%%%%%%%

%%%%%%%%%%%%%%%%%%%%%%%%%%%%%%%%%%%%%%%%%%%%%%%%%%%%%%%%%%%%%%%%%%%%%%%%%%%%%%%%%
\begin{theorem}
 The complement of a weakly balanced graph in $\R^n$ is $(n{-}2)$-convex.
\end{theorem}
%%%%%%%%%%%%%%%%%%%%%%%%%%%%%%%%%%%%%%%%%%%%%%%%%%%%%%%%%%%%%%%%%%%%%%%%%%%%%%%%%

%%%%%%%%%%%%%%%%%%%%%%%%%%%%%%%%%%%%%%%%%%%%%%%%%%%%%%%%%%%%%%%%%%%%%%%%%%%%%%%%%
\begin{proof}
 We must show that for any hyperplane $L$ that meets $\scrC$, the map
\[
   \iota_{n-2}\ \colon\ 
    \witi{H}_{n-2}(L\cap\scrC^c)\ \longrightarrow\ \witi{H}_{n-2}(\scrC^c)
\]
 is injective.
  
 The reduced homology group $\witi{H}_{n-2}(L\cap\scrC^c)$ is free with one generator for each
 bounded connected component of $L\cap\scrC$.
 We assume that $\scrC\not\subset L$, for otherwise there are no bounded connected components of $L\cap\scrC$ and
 $\witi{H}_{n-2}(L\cap\scrC^c)$  vanishes.

 By Lemma~\ref{L:homology_basis}, if $\scrC_1,\dotsc,\scrC_m$ are the bounded connected components of
 $L\cap\scrC$, then there are open subsets $\beta_1,\dotsc,\beta_m$ with disjoint closures and
 $\beta_i\supset\scrC_i$ such that if $\gamma_i:=\partial\beta_i$ for $i=1,\dotsc,m$, then the cycle classes
 $[\gamma_i]$ are a basis for $\witi{H}_{n-2}(L\cap\scrC^c)$.

 Suppose that $\zeta\in\witi{H}_{n-2}(L\cap\scrC^c)$ with 
 $\iota_{n-2}(\zeta)=0$ in  $\witi{H}_{n-2}(\scrC^c)$.
 There are unique integers $b_1,\dotsc,b_m$ such that $\zeta=[\gamma]$ where $\gamma$ is the cycle
\[
    \gamma\ =\ \sum_{i=1}^m b_i \,\gamma_i\,.
\]
 Since $\iota_{n-2}(\zeta)=0$, there is a $(n{-}1)$-chain \defcolor{$Z$} in $\scrC^c$ with 
 $\partial Z=\gamma$.
 If we set 
\[
   \beta\ :=\  -\sum_{i=1}^m b_i \beta_i\,,
\]
 then $Z+\beta$ is closed in $\R^n$, and hence bounds an $n$-chain $D$ in $\scrC^c$.

 Suppose that $\zeta\neq 0$, so that some coefficient $b_i$ is nonzero.
 Then $-b_i\gamma_i$ forms part of the boundary of $D$, and we conclude that for every point $p$ of
 the interior of $\beta_i$, $D$ contains a neighborhood of $p$ in one of the halfspaces defined by the
 hyperplane $L$.

 Since $\beta_i\cap\scrC$ is nonempty and contained in the interior of $\beta_i$, and $\scrC\not\subset L$,
 there is a point $p$ of $\beta_i\cap\scrC$ lying in the closure of $\scrC\smallsetminus(L\cap\scrC)$.
 If $w$ is a vector normal to $L$, then there is an edge $e$ of $\scrC$ incident to $p$ with $w\cdot v_e\neq 0$.
 Possibly replacing $e$ by another edge incident to $p$ and $w$ by $-w$, we may assume that $w\cdot v_e>0$ and
 that $e$ meets the interior of $D$.
 By Lemma~\ref{L:unboundedPath}, there is a path $\gamma\colon[0,\infty)\to\scrC$ with $\gamma(0)=p$,
 $w\cdot\gamma(t)$ an unbounded and increasing function of $t$, and whose image contains the edge $e$.
 Observe that $L$ is the set of points $x$ with $w\cdot x=w\cdot p$.

 By our choice of $e$ and $w$, there is some $t_0>0$ such that $\gamma(t_0)$ lies in the interior of the
 $n$-chain $D$.
 Since $w\cdot\gamma(t)$ is unbounded, but $D$ is bounded, there is a point $t_1>t_0$ with $\gamma(t_1)$ lying
 on the boundary $Z +\beta$ of $D$.
 As $w\cdot\gamma(t_1)>w\cdot\gamma(0)$ and $\gamma(0)=p$, we see that $\gamma(t_1)\not\in L$, and therefore
 is a point of $Z$.
 But this implies that $\scrC\cap Z\neq \emptyset$, contradicting that $Z$ is a chain in
 $\scrC^c$.
 We conclude that $\zeta=0$, which implies that $\iota_{n-2}$ is injective.
\end{proof}
%%%%%%%%%%%%%%%%%%%%%%%%%%%%%%%%%%%%%%%%%%%%%%%%%%%%%%%%%%%%%%%%%%%%%%%%%%%%%%%%%

%%%%%%%%%%%%%%%%%%%%%%%%%%%%%%%%%%%%%%%%%%%%%%%%%%%%%%%%%%%%%%%%%%%%%%%%%%%%%%%%%
\section{Complex nonarchimedean amoebas}\label{S:three}
%%%%%%%%%%%%%%%%%%%%%%%%%%%%%%%%%%%%%%%%%%%%%%%%%%%%%%%%%%%%%%%%%%%%%%%%%%%%%%%%%

Let $\K$ be the field of complex Puiseux series in the variable $s$ with $\nu(s)=1$.
We use Jonsson's Theorem to study the complement of the
nonarchimedean amoeba \defcolor{$\scrT$} of a variety $V$ in $(\K^\times)^n$.
By Proposition~\ref{P:limit}, there is a smooth curve $C$, a point $o\in C$, a local parameter $u$ at $o$, and a
family $\calV\subset(C\smallsetminus\{o\})\times(\C^\times)^n$ with
\[
   \lim_{a\to o} \frac{-1}{\log|u(a)|} \scrA(\calV_a)\ =\ \scrT\,.
\]
If $V$ has codimension $r$, then $\calV$ and its fibers have codimension $r$, and if in addition $V$ is a complete
intersection, then so are the fibers $\calV_a$ of $\calV$.

Let \defcolor{$a(t)$} for $t$ positive and near $0$ be the analytic arc defined by $u(a(t))=t$.
If we set 
 \begin{equation}\label{Eq:scaled_amoebas}
   \defcolor{\scrA_t}\ :=\  \frac{-1}{\log t} \scrA(\calV_{a(t)})\,,
 \end{equation}
then we have 
 \[
   \lim_{t\to 0} \scrA_t\ =\ \scrT\,.
 \]
In particular, for every point $x\in\scrT$ and every $t$ sufficiently small, there is a point of $\scrA_t$ close to
$x$.  
We need more, that these points of $\scrA_t$ lie in a given affine $(k{+}1)$-plane
through $x$.
The following technical lemma, whose proof we defer, guarantees this.
It uses the weak higher convexity of the scaled amoebas $\scrA_t$.

%%%%%%%%%%%%%%%%%%%%%%%%%%%%%%%%%%%%%%%%%%%%%%%%%%%%%%%%%%%%%%%%%%%%%%%%%%%%%
\begin{lemma}\label{L:technical}
 Suppose that $L$ is an affine $(k{+}1)$-plane that is transverse to $\scrT$.
 For every $\epsilon>0$, there is a $\delta>0$ such that for every $x\in L\cap\scrT$ and $0<t<\delta$ we have
\[
   d(x\,,\, L\cap \scrA_t)\ <\ \epsilon\,.
\]
\end{lemma}
%%%%%%%%%%%%%%%%%%%%%%%%%%%%%%%%%%%%%%%%%%%%%%%%%%%%%%%%%%%%%%%%%%%%%%%%%%%%%

We deduce our main result.

%%%%%%%%%%%%%%%%%%%%%%%%%%%%%%%%%%%%%%%%%%%%%%%%%%%%%%%%%%%%%%%%%%%%%%%%%%%%%
\begin{theorem}\label{T:main}
 The complement $\scrT^c$ of the nonarchimedean amoeba $\scrT$ of a variety $V\subset(\K^\times)^n$ of codimension
 $k{+}1$ is weakly $k$-convex.
 If $V$ is a complete intersection, then the complement is $k$-convex.
\end{theorem}
%%%%%%%%%%%%%%%%%%%%%%%%%%%%%%%%%%%%%%%%%%%%%%%%%%%%%%%%%%%%%%%%%%%%%%%%%%%%%

%%%%%%%%%%%%%%%%%%%%%%%%%%%%%%%%%%%%%%%%%%%%%%%%%%%%%%%%%%%%%%%%%%%%%%%%%%%%%
\begin{remark}
 Both Lemma~\ref{L:technical} and Theorem~\ref{T:main} use only that $\scrT$ is a polyhedral complex of pure
 codimension $k{+}1$ that is the Kuratowski limit of a family $\{\scrA_t\mid t>0\}$ whose complements are (weakly)
 $k$-convex.
 Thus we have proven a stronger result about polyhedral complexes that are limits of sets whose complements are
 $k$-convex. 
\end{remark}
%%%%%%%%%%%%%%%%%%%%%%%%%%%%%%%%%%%%%%%%%%%%%%%%%%%%%%%%%%%%%%%%%%%%%%%%%%%%%

%%%%%%%%%%%%%%%%%%%%%%%%%%%%%%%%%%%%%%%%%%%%%%%%%%%%%%%%%%%%%%%%%%%%%%%%%%%%%
\begin{proof}[Proof of Theorem~$\ref{T:main}$]
 It suffices to use affine $(k{+}1)$-planes $L$ that meet $\scrT$
 transversally to test the $k$-convexity of $\scrT^c$.
 Let $L$ be an oriented affine $(k{+}1)$-plane that meets $\scrT$ transversally and let
 $\zeta\in\witi{H}_k(L\cap\scrT^c)$ be a cycle.

 Let $\{x_p\mid p\in\Pi\}$ be the finite set $L\cap\scrT$.
 Let $\epsilon>0$ be such that if \defcolor{$\beta_p$} is the closed ball of radius $\epsilon$ in $L$ centered at
 $x_p$ for $p\in\Pi$, then these balls are disjoint.
 If $\defcolor{\gamma_p}:=\partial\beta_p$ for $p\in\Pi$, then the cycle classes of 
 $\{\gamma_p\mid p\in\Pi\}$ are a basis for $\witi{H}_k(L\cap\scrT^c)$ and they span its cone of positive cycles.
 Thus there is an integer combination $\gamma$ of the $\gamma_p$ with $[\gamma]=\zeta$.
 Suppose that $\iota_k(\zeta)=0$ in $\witi{H}_k(\scrT^c)$.
 Then there is a $(k{+}1)$-chain $Z$ in $\scrT^c$ with $\partial Z=\gamma$.

 Let $\scrA_t$ for $t$ positive and near $0$ be a family of scaled amoebas~\eqref{Eq:scaled_amoebas} of varieties
 of codimension $k{+}1$ that converges to $\scrT$.
 As $Z$ is compact, Lemma~\ref{L:compacta} implies that there is some $\delta>0$ such that $Z$ is disjoint from
 $\scrA_t$ for all $0<t<\delta$. 
 As $[\partial Z]=\zeta$, we conclude that $\iota_k(\zeta)=0$ in $\witi{H}_k(\scrA_t^c)$ for any 
 $0<t<\delta$.

 By Lemma~\ref{L:technical}, after possibly shrinking $\delta$, if $p\in\Pi$ and $0<t<\delta$, then 
 $d(x_p,L\cap\scrA_t)<\epsilon$.
 Thus for all $0<t<\delta$, the scaled amoeba $\scrA_t$ meets each ball $\beta_p$ and each sphere $\gamma_p$ is
 disjoint from $\scrA_t$, as $\gamma_p\subset Z$.
 In particular, the cycle classes $[\gamma_p]$ for $p\in\Pi$ are linearly independent in
 $\witi{H}_k(L\cap\scrA_t^c)$ and they span a subset of its positive cone.

 By Henriques' Theorem~\cite[Thm.~4.1]{He04}, the complement $\scrA_t^c$ of the scaled amoeba is weakly $k$-convex.
 Since $\iota_k(\zeta)=0$, $\zeta$ is not in the positive cone of $\witi{H}_k(L\cap\scrA_t^c)$ and so it is not a
 positive integer combination of the $[\gamma_p]$. 
 Consequently, $\zeta$ is not a positive class of $\witi{H}_k(L\cap\scrT^c)$.
 This shows that $\scrT^c$ is weakly $k$-convex.

 Now suppose that $V$ is a complete intersection.
 By Proposition~\ref{P:limit}, we may assume that $\scrA_t$ is the amoeba of a complete intersection, and so by the
 Theorem of Bushueva and Tsikh~\cite{BT12}, $\scrA_t^c$ is $k$-convex.
 As $\iota_k(\zeta)=0$ in  $\witi{H}_k(\scrA_t^c)$, we conclude that $\zeta=0$ in $\witi{H}_k(L\cap\scrA_t^c)$.
 Since the cycle classes $[\gamma_p]$ are linearly independent in both 
 $\witi{H}_k(L\cap\scrA_t^c)$ and $\witi{H}_k(L\cap\scrT)$, we conclude that 
 $\zeta=0$ in $\witi{H}_k(L\cap\scrT)$, and so $\scrT^c$ is $k$-convex.
\end{proof}
%%%%%%%%%%%%%%%%%%%%%%%%%%%%%%%%%%%%%%%%%%%%%%%%%%%%%%%%%%%%%%%%%%%%%%%%%%%%%

%%%%%%%%%%%%%%%%%%%%%%%%%%%%%%%%%%%%%%%%%%%%%%%%%%%%%%%%%%%%%%%%%%%%%%%%%%%%%
\begin{proof}[Proof of Lemma~$\ref{L:technical}$]
 The tropical variety $\scrT$ admits the structure of a finite polyhedral complex of dimension $n{-}k{-}1$.
 Since $L$ is has dimension $(k{+}1)$ and is transverse to $\scrT$, it meets $\scrT$ only in polyhedra of maximal
 dimension.
 Let $\{\sigma_p\mid p\in\Pi\}$ be the set of maximal polyhedra of $\scrT$ meeting $L$.
 If we set $\defcolor{x_p}:=L\cap\sigma_p$, then $\{x_p\mid p\in\Pi\}=L\cap\scrT$.
 Since $\scrT\cap L$ is finite, it suffices to prove the lemma for a single point $x_p\in L\cap\scrT$.

 Shrinking $\epsilon$ if necessary, we may assume that if $\sigma$ is a polyhedron of $\scrT$, then
\begin{enumerate}
 \item If $\sigma$ does not meet $L$, then $d(L,\sigma)>2\epsilon$,
 \item If $\sigma\neq \sigma_p$, then $d(x_p,\sigma)>2\epsilon$.
\end{enumerate}

 A plane $L'$ is \demph{parallel} to $L$ if $L'=L+v$ for some $v\in\R^n$.
 By Assumption (1), if $L'$ is parallel to $L$ with $d(L',L)\leq 2\epsilon$, then $L'$ meets $\sigma_p$
 transversally in a point \defcolor{$x'_p$} in the relative interior of $\sigma_p$.
 Let \defcolor{$\theta$} be the minimum angle between $\sigma$ and $L$.
 Then $\pi/2\geq\theta>0$, and if $L'$ is parallel to $L$ with $d(L',L)\leq\rho$ with $\rho\leq 2\epsilon$, then
 $d(x_p',x_p)<\rho/\sin\theta$. 
 Let $\defcolor{\beta}\subset L$ be the ball of radius $\epsilon$ centered at $x_p$ and
 $\defcolor{\gamma}=\partial\beta$ the corresponding sphere.
 Observe that $d(\gamma,\sigma_p)=\epsilon\sin\theta$.

 Let \defcolor{$Z$} be the union of translates $\gamma+(z-x_p)$ where $z\in\sigma_p$ with $d(z,x_p)\leq\epsilon$.
 Then $Z$ is compact and it lies in the closed ball centered at $x_p$ of radius $2\epsilon$.
 By construction, $Z\cap\sigma_p=\emptyset$, and by Assumption (2), $Z$ meets no other polyhedron of $\scrT$, and
 is therefore disjoint from $\scrT$.
 By the definition of Kuratowski limit and Lemma~\ref{L:compacta}, there is a $\delta>0$ such that if $0<t<\delta$
 then $d(x,\scrA_t)<\frac{1}{2}\epsilon\sin\theta$ and $\scrA_t$ is disjoint from $Z$.

 Fix a positive $t<\delta$.
 Let $y'\in\scrA_t$ be a point with $d(x_p,y')<\frac{1}{2}\epsilon\sin\theta$.
 Set $\defcolor{L'}:=L+(y'-x_p)$, a plane parallel to $L$ with $d(L',L)<\frac{1}{2}\epsilon\sin\theta$.
 Set $\defcolor{x'_p}:=L'\cap\sigma_p$, then $d(x'_p,x_p)<\epsilon/2$.
 Let $\defcolor{\Gamma}\subset Z$ be the cylinder
\[
   \bigcup\{ \gamma + t(x'_p-x_p)\;\mid\; t\in[0,1]\}\,.
\]

 Observe that
\[
   d(x_p',y')\ \leq\ d(x'_p,x_p)\ +\ d(x_p,y')\ <\ 
    \tfrac{1}{2}\epsilon + \tfrac{1}{2}\epsilon\sin\theta\ \leq\ \epsilon\,.
\]
 Thus $y'\in\beta'$, where $\defcolor{\beta'}=\beta+(x_p'-x_p)$ is the ball in $L'$ centered at $x_p'$ of radius
 $\epsilon$. 
 Since $\defcolor{\gamma'}:=\partial\beta'\subset Z$, it is disjoint from $\scrA_t$, and so we have that 
\[
   0\ \neq\ [\gamma']\ \in\ \witi{H}_k(L'\cap\scrA_t^c)
\]
 is a positive cycle (we fix an orientation of $L'$).
 By Henriques' Theorem~\cite[Thm.~4.1]{He04}, $\scrA_t^c$ is weakly $k$-convex, and so 
\[
  \iota_k[\gamma']\ \in\ \witi{H}_k(\scrA_t^c)
\]
 is also non-zero.
 As $\gamma'$ is the boundary of the cycle $\Gamma \cup\beta$, we canot have that $\Gamma\cup\beta$ lies
 in the complement $\scrA_t^c$, and so $(\Gamma\cup\beta)\cap\scrA_t\neq\emptyset$.
 Since $\Gamma\subset\scrA_t^c$, we conclude that $\beta$ meets $\scrA_t$.
 As the boundary $\gamma$ of $\beta$ is disjoint from $\scrA_t$, there is a point $y\in\scrA_t$ in the
 interior of $\beta$, which proves the lemma.
\end{proof}
%%%%%%%%%%%%%%%%%%%%%%%%%%%%%%%%%%%%%%%%%%%%%%%%%%%%%%%%%%%%%%%%%%%%%%%%%%%%%

\section*{Acknowledgements}
We thank Mattias Jonsson and Eric Katz for their help in understanding their work.
%%%%%%%%%%%%%%%%%%%%%%%%%%%%%%%%%%%%%%%%%%%%%%%%%%%%%%%%%%%%%%%%%%%%%%%%%%%%%

%%%%%%%%%%%%%%%%%%%%%%%%%%%%%%%%%%%%%%%%%%%%%%%%%%%%%%%%%%%%%%%%%%%%%%%%%%%%%
\def\cprime{$'$}
\providecommand{\bysame}{\leavevmode\hbox to3em{\hrulefill}\thinspace}
\providecommand{\MR}{\relax\ifhmode\unskip\space\fi MR }
\providecommand{\MRhref}[2]{%
  \href{http://www.ams.org/mathscinet-getitem?mr=#1}{#2}
}
\providecommand{\href}[2]{#2}

\end{document}